\documentclass[a4paper,12pt]{amsart}
\usepackage{amsfonts,amssymb,amsmath}
\usepackage[utf8]{inputenc}

\newtheorem{teo}{Theorem}[section]
\newtheorem{lema}{Lemma}[section]
\newtheorem{pro}{Proposition}[section]
\newtheorem{defi}{Definition}[section]
\newtheorem{rema}{Remark}[section]
\newtheorem{coro}{Corollary}[section]
\numberwithin{equation}{section}

\def\d {\mathrm{d}}

\def \x {\times}

\def \T{\rm{T}}

\def \del{\partial}
\def \z{\zeta}

\def \End {\mathrm{End}}

\def\f{\varphi}
\def\dr{\partial_r}

\def\i{\lrcorner}
\def \RM{\mathbb{R}}
\def \TM{\mathbb{T}}

\def\r{\end{proof}}
\def\SU{\mathrm{SU}}
\def\Sp{\mathrm{Sp}}
\def\vol{\mathrm{vol}}
\def\Id{\mathrm{Id}}
\def\beq{\begin{equation}}
\def\eeq{\end{equation}}
\def\bea{\begin{eqnarray*}}
\def\eea{\end{eqnarray*}}
\def\n{\nabla}
\def\psp{\psi ^+}
\def\psm{\psi ^-}
\def\tr{\mathrm{tr}}
\def\pp{\mathfrak{p}}
\def\kk{\mathfrak{k}}
\def\SU{\mathrm{SU}}
\def\su{\mathfrak{su}}
\def\gg{\mathfrak{g}}
\def\Hess{{\rm Hess}}

\def\sideremark#1{\ifvmode\leavevmode\fi\vadjust{\vbox to0pt{\vss
 \hbox to 0pt{\hskip\hsize\hskip1em
 \vbox{\hsize2.5cm\tiny\raggedright\pretolerance10000
 \noindent #1\hfill}\hss}\vbox to8pt{\vfil}\vss}}}%

\begin{document}
\title{Toric Nearly K\"ahler manifolds}
\author[A. Moroianu]{ Andrei Moroianu}
\address{Andrei Moroianu, Laboratoire de Math\'ematiques d'Orsay, Univ.\ Paris-Sud, CNRS,
Universit\'e Paris-Saclay, 91405 Orsay, France}
\email{andrei.moroianu@math.cnrs.fr}
\author[P.-A.Nagy]{Paul-Andi Nagy}
\address{Paul-Andi Nagy, Department of Mathematics, The University of Murcia, Campus de Espinardo, E-30100 Espinardo, Murcia, Spain}
\email{npaulandi@gmail.com}

\begin{abstract} 
We show that 6-dimensional strict nearly K\"ahler manifolds admitting
effective $\TM^3$ actions by automorphisms are completely characterized
in the neigbourhood of each point by a function on $\RM^3$ satisfying
a certain Monge--Amp\`ere-type equation.
\end{abstract}
\maketitle

\section{Introduction}

Nearly Kähler manifolds were originally introduced as the class $\mathcal{W}_1$ in the Gray-Hervella classification of almost Hermitian manifolds \cite{gh}. More precisely, an almost Hermitian manifold $(M,g,J)$ is called nearly Kähler (NK in short) if $(\nabla_XJ)(X)=0$ for every vector field $X$ on $M$, where $\nabla$ denotes the Levi-Civita covariant derivative of $g$. A NK manifold is called {\em strict} if $\nabla J\ne 0$.

In \cite{n} it was shown that every NK manifold is locally a product of one of the following types of factors:
\begin{itemize}
\item Kähler manifolds;
\item 3-symmetric spaces;
\item twistor spaces of positive quaternion-Kähler manifolds;
\item 6-dimensional strict NK manifolds.
\end{itemize}

It is thus crucial to understand the 6-dimensional case, to which we will restrict in the sequel. In dimension 6, strict NK are important for several further reasons: they admit real Killing spinors \cite{fg}, in particular they are Einstein with positive scalar curvature, and they can be characterized in terms of exterior differential systems as manifolds with special generic 3-forms in the sense of Hitchin \cite{hit}.

Until 2015, the only known examples of compact 6-dimensional strict NK manifolds were the 3-symmetric spaces $S^6=\mathrm{G}_2/\SU(3)$, $F(1,2)=\SU_3/S^1\times S^1$,
$CP^3=\Sp_2/S^1\x \Sp_1$ and $S^3\x S^3=\Sp_1\x\Sp_1\x\Sp_1/\Sp_1$. Moreover, J.-B. Butruille has shown in \cite{jbb} that these are the only homogeneous examples.

A breakthrough was achieved very recently by L. Foscolo and M. Haskins, who studied cohomogeneity one NK metrics and obtained the first examples of non-homogeneous NK structures on $S^6$ and $S^3\x S^3$, cf. \cite{fh}, \cite{f}. The corresponding metrics are shown to exist but cannot be constructed explicitly. However, their isometry group is known, and is equal to $\SU(2)\times\SU(2)$ in both cases.

 It is easy to show that a torus acting by automorphisms of a NK structure $(M^6,g,J)$ has dimension at most 3 (Corollary \ref{c35}), and if equality holds, then the corresponding commuting vector fields span a totally real distribution on a dense open set of $M$ (cf. Lemma \ref{2.3}). In the present paper we study 6-dimensional nearly Kähler manifolds whose automorphism group has maximal possible rank. We call them {\em toric NK structures} by analogy with the Kähler case.

Our main result is to give a local characterization of toric NK structures in terms of a single function of 3 real variables satisfying to a certain Monge--Amp\`ere-type equation. We conjecture that the only compact toric NK manifold is $S^3\x S^3$ with its 3-symmetric NK structure.

\section{Structure equations}
Let $M^6$ be an oriented manifold. An $\SU(3)$-structure on $M$ is a
triple $(g,J,\psi)$, where $g$ is a Riemannian metric, $J$ is a
compatible almost complex structure (i.e. $\omega:=g(J\cdot,\cdot)$ is
a 2-form), and $\psi=\psi ^++i\psi ^-$ is a
$(3,0)$ complex volume form satisfying 
\beq\label{vo}\psi\wedge\bar\psi={-8i}\vol_g.
\eeq

Following Hitchin \cite{hit}, it is possible to characterize
$\SU(3)$-structures in terms of exterior forms only. 
If $\psi ^+$ is a three-form on $M$,
one can define $K\in \End (\T M)\otimes \Lambda ^6M$ by
$$X\mapsto K(X):=(X\i\psi ^+)\wedge\psi ^+\in \Lambda ^5M\simeq
\T M\otimes \Lambda ^6M.$$

\begin{lema}\label{1.1} {\rm (\cite{hit})} A non-degenerate $2$-form $\omega$ on $M$, and a $3$-form $\psi
^+\in \Lambda ^3M$ satisfying 
\begin{enumerate}
\item[(i)] $\omega\wedge\psi ^+=0$.
\item[(ii)] $\tr K^2=-\frac16 (\omega ^3)^2\in (\Lambda ^6M)^{\otimes 2}$.
\item[(iii)] $\omega(X,K(X))/\omega^3>0$ for every $X\ne 0$.
\end{enumerate}
define an $\SU(3)$-structure on $M$.
\end{lema}
\begin{proof}
It is easy to check that 
\beq\label{trk}K^2=\frac16 \Id\otimes\tr(K^2)\in \End (\T M)\otimes (\Lambda ^6M)^{\otimes 2}.
\eeq
From (ii) we see that $J:=6K/\omega ^3$ is an almost complex structure
on $M$. The tensor $g$ defined by $g(\cdot,\cdot):=\omega(\cdot,J\cdot)$
is symmetric by (i) and positive definite by (iii). Finally, it is straightforward to
check that $\psi ^++i\psi ^-$ is a $(3,0)$ complex volume form
satisfying \eqref{vo}, where $\psi^{-}:=-\psi^{+}(J\cdot,
\cdot, \cdot)$. 
\r

Since $\vol_g=\frac16\omega ^3$, \eqref{vo} is equivalent to 
\begin{equation} \label{comp}
\psi^{+} \wedge \psi^{-}=\frac{2}{3} \omega^3.
\end{equation}

\begin{defi} A strict NK structure on $M^6$ is
an $\SU(3)$-structure $(\psi^{\pm}, \omega)$ satisfying  
\begin{equation} \label{nk1}
\d \omega=3 \psi^{+}
\end{equation}
and
\begin{equation} \label{nk2}
\d \psi^{-}=-2 \omega \wedge \omega.
\end{equation}
\end{defi}

For an alternative definition and more details on NK manifolds we refer
to \cite{gray} or \cite{mns1}.

Let $g$ denote the Riemannian metric induced by $(\psi^{\pm},
\omega)$, with Levi-Civita covariant derivative $\n$, and let $J$ denote
the induced almost complex structure. From now on we identify vectors
and 1-forms, as well as skew-symmetric endomorphisms and 2-forms using
$g$.

We then have the relations (cf. \cite{mns1}):
\beq JX\i\psp=(X\i\psp)\circ J=- J\circ (X\i\psp),\qquad\forall X\in \T M,
\eeq
\beq\label{np}\n_XJ=X\i \psp,\qquad\forall X\in \T M.
\eeq

\bigskip

\section{Torus actions by automorphisms}
\medskip

Suppose that $(M^6,\psi^{\pm}, \omega,g,J)$ is a strict NK structure
carrying a toric action by automorphisms. More precisely, we assume that there exists some positive integer $d\ge1$ and 
$k$ linearly independent Killing vector fields $\z_i,\ 1 \le i \le d$ such that
$[\z_i,\z_j]=0$ for $1\le i,j\le d$, which are pseudo-holomorphic in the sense  
that $L_{\z_i}J=0$ for $1\le i\le d$. This last condition is equivalent with the requirement that  
\begin{equation}\label{kh}
L_{\z_i} \psi^{\pm}=0,\ L_{\z_i} \omega=0,\qquad 1 \le i \le d.
\end{equation}
Notice that if $M$ is compact and not isometric with the
standard sphere, \eqref{kh} follow directly from the Killing condition
(cf. \cite{mns1}, Proposition 3.1). 

We define the smooth functions  
$\mu_{ij}$ on $M$ by setting $\mu_{ij}:=\omega(\z_i, \z_j)$. 
\begin{lema} \label{start}
The following relations hold for every $i,j,k\in\{1,\ldots,d\}$:
\begin{itemize}
\item[(i)] $\d \mu_{ij}=-3\z_i \lrcorner \z_j \lrcorner \psi^{+}$.
\item[(ii)] $\psi^{+}(\z_i, \z_j, \z_k)=0$.
\item[(iii)] $[\z_i, J\z_j]=0$.
\item[(iv)] $[J\z_i, J\z_j]=4(J\z_j\i\z_i\i\psp)^\sharp$.
\end{itemize}
\end{lema}
\begin{proof}
(i) From \eqref{nk1} together with the Cartan formula we get 
\begin{equation*}
0=L_{\z_j} \omega=\z_j \lrcorner \d \omega+\d (\z_j \lrcorner \omega)=
3\z_j \lrcorner \psi^{+}+\d (\z_j \lrcorner \omega).
\end{equation*}
Taking now the interior product with $\z_i$ yields 
\begin{equation*}
0=3\z_i \lrcorner \z_j \lrcorner \psi^{+}+\z_i \lrcorner \d (\z_j
\lrcorner \omega) 
\end{equation*}
and the claim follows by taking into account that 
$$\z_i \lrcorner
\d (\z_j \lrcorner \omega)=L_{\z_i}(\z_j \lrcorner \omega)-\d (\z_i
\lrcorner \z_j \lrcorner \omega)= 
\d \mu_{ij}.$$

(ii) Using (i) we can write
$$\psi^{+}(\z_i, \z_j, \z_k)=-\frac13\d \mu_{jk}(\z_i)= 
-\frac13L_{\z_i}(\omega(\z_j,\z_k))=0.$$

(iii) Follows directly from $L_{\z_i}J=0$ and the fact that the $\z_i$'s
mutually commute. 

(iv) On every almost Hermitian manifold, the Nijenhuis tensor 
$$N(X,Y):=[X,Y]+J[X,JY]+J[JX,Y]-[JX,JY]$$
can be expressed as
\beq \label{l1}N(X,Y)=J(L_XJ)Y-(L_{JX}J)Y
\eeq
for all vector fields $X,Y$.
On the other hand, \eqref{np} shows that on every NK manifold, the
Nijenhuis tensor satisfies 
\beq \label{l2}
N(X,Y)=J(\n_XJ)Y-J(\n_YJ)X-(\n_{JX}J)Y+(\n_{JY}J)X=-4Y\i JX\i\psp.
\eeq
Applying \eqref{l1} and \eqref{l2} to $X=\z_i$, and using
the fact that $L_{\z_i}J=0$ yields 
\beq\label{lz}(L_{J\z_i}J)=4J\z_i\i\psp.\eeq
This, together with (iii), finishes the proof.
\end{proof} 

\begin{lema}\label{??} If $\xi$ is a Killing vector field, $J\xi$
  cannot be Killing on any open set $U$. 
\end{lema}
\begin{proof}
From Corollary 3.3 and Lemma 3.4 in \cite{mns1} we have 
$$(\d J\xi)^{(2,0)}=\d J\xi=-\xi\i \d \omega=-3\xi\i\psp$$
and 
$$(\d \xi)^{(2,0)}=-J\xi\i\psp$$
for every Killing vector field $\xi$. If $J\xi$
were Killing on some open set, the same relations applied to $J\xi$
would read
$$(\d \xi)^{(2,0)}=3J\xi\i\psp$$
and 
$$(\d J\xi)^{(2,0)}=\xi\i\psp,$$
a contradiction.
\r

Assume from now on that the dimension of the torus acting by automorphisms satisfies $d\ge 2$. 

\begin{lema} \label{le1} For every $i\ne j$ in $\{1,\ldots,d\}$, the vector fields $\{\z_i,\z_j,J\z_i,J\z_j\}$
  are linearly independent on a dense open subset of $M$.
\end{lema}

\begin{proof} One can of course assume $i=1,j=2$. 
If the contrary holds, there exists some open set $U$ on which $\z_1$
does not vanish and functions
$a,b:U\to \RM$ such that 
\beq\label{ab}\z_2=a\z_1+bJ\z_1.\eeq
We differentiate this
relation on $U$ with 
respect to the Levi-Civita covariant derivative $\n$ and obtain the
following relation between endomorphisms of $\T M$:
\bea \n\z_2&=&\d a\otimes \z_1+a\n \z_1+\d b\otimes J\z_1+b\n J\z_1\\
&=&\d a\otimes \z_1+a\n \z_1+\d b\otimes J\z_1-b\z_1\i\psp+bJ\circ(\n\z_1).
\eea
Taking the symmetric parts in this equation yields
$$0=\d a\odot \z_1+\d b\odot J\z_1+b(J\circ(\n\z_1))^{sym}.$$
Since $\n\z_1$ is skew-symmetric, $(J\circ(\n\z_1))^{sym}$ commutes
with $J$, whence $J$ commutes with $\d a\odot \z_1+\d b\odot J\z_1$. On
the other hand,  $J$ commutes with $\d a\odot \z_1+J\d a\odot J\z_1$,
thus it commutes with $(\d b-J\d a)\odot J\z_1$. This implies
$\d b=J\d a$. Differentiating this again with respect to $\n$ yields
$$\n \d b=\n(J\d a)=-\d a\i\psp+J\circ\n \d a.$$
Taking the skew-symmetric part in this equality shows that
$$\d a\i\psp=(J\circ\n \d a)^{skew}.$$
But the left hand side anti-commutes with $J$, whereas the right hand
side commutes with $J$ (since $\nabla\d a$ is symmetric). Thus $\d a=0$, so $a$
and $b$ are constants. From \eqref{ab}, we obtain that $J\z_1$ is a
Killing vector field on $U$, which is impossible by Lemma
\ref{l1}. This contradiction concludes the proof.
\r

\begin{coro}\label{c1} The vector fields
  $\{\z_1,\z_2,J\z_1,J\z_2,\z_1\i\z_2\i\psp,J\z_1\i\z_2\i\psp\}$
  are linearly independent on a dense open subset of $M$.
\end{coro}
\begin{proof}
This follows from Lemma \ref{le1} using the fact that the vectors
$\z_1\i\z_2\i\psp$ and $J\z_1\i\z_2\i\psp$ are orthogonal to
$\z_1,\z_2,J\z_1$ and $J\z_2$, and they are both non-vanishing at each
point where  $\{\z_1,\z_2,J\z_1,J\z_2\}$ are linearly independent.
\r

From now on we assume that $d\ge 3$. 

\begin{lema} \label{2.3} For every mutually distinct $1\le i,j,k\le d$, the $6$ vector fields $\z_i,$ $ \z_j,$ $\z_k$, $J\z_i,$
$J\z_j,$ $J\z_k$ are linearly independent on a dense open subset $M_0$  of $M$.
\end{lema}

\begin{proof}
We may assume that $i=1$, $j=2$ and $k=3$. Like before, if the statement does not hold, there exists some open
set $U$ on which $\z_1$ 
does not vanish and functions
$a_1,b_1,a_2,b_2:U\to \RM$ such that 
\beq\label{abc}\z_3=a_1\z_1+b_1J\z_1+a_2\z_2+b_2J\z_2.\eeq
 By Lemma \ref{le1}, one may assume that $\{\z_1,\z_2,J\z_1,J\z_2\}$
  are linearly independent on $U$. Taking the Lie derivative with
  respect to $J\z_1$ in \eqref{abc} and using Lemma \ref{start} (iii)
  and (iv) yields
$$0=J\z_1(a_1)\z_1+J\z_1(b_1)J\z_1+J\z_1(a_2)\z_2+
J\z_1(b_2)J\z_2+4b_2J\z_2\i\z_2\i\psp.
$$
From Corollary \ref{c1} we get $b_2=0$. Similarly, taking the Lie
derivative with respect to $J\z_2$ in \eqref{abc} we get
$b_1=0$. Therefore  \eqref{abc} becomes 
\beq\label{abc1}\z_3=a_1\z_1+a_2\z_2.\eeq
Differentiating this equation with respect to $\n$ and taking the
symmetric part yields
$$0=\d a_1\odot\z_1+\d a_2\odot\z_2.$$
Since $\z_1$ and $\z_2$ are linearly independent on $U$, this implies 
$\d a_1=c\z_2$ and $\d a_2=-c\z_1$ for some function $c:U\to\RM$. On the
other hand, taking the Lie derivative with respect to $\z_2$ in \eqref{abc1} 
yields $0=\z_2(a_1)\z_1+\z_2(a_2)\z_2$, thus $\z_2(a_1)=0$, so
finally $c|\z_2|^2=g(\d a_1,\z_2)=\z_2(a_1)=0$, whence $c=0$. This shows
that $a_1$ and $a_2$ are constant, contradicting the hypothesis that
$\z_1,\z_2$ and $\z_3$ are linearly independent Killing vector
fields. This proves the lemma.
\r

\begin{coro}\label{c35} The rank $d$ of the automorphism group of $M$ is at most $3$.
\end{coro}
\begin{proof}
Assume for a contradiction that $d\ge 4$, so there exist 4 linearly independent mutually commuting Killing vector fields $\z_1,\ldots,\z_4$ on $M$ preserving the almost complex structure $J$. From Lemma \ref{2.3}, there exist functions $a_i$ and $b_i$ ($i=1,2,3$) on $M_0$ such that 
\beq\label{1234}\z_4=\sum_{j=1}^3a_j\z_j+b_jJ\z_j.\eeq
From Lemma \ref{start} (ii) we get $\psi^{+}(\z_1, \z_2, \z_3)=\psi^{+}(\z_1, \z_2, \z_4)=0$. Using \eqref{1234} together with the fact that $\psi^+(X,JX,\cdot)=0$ for every $X$, we get $b_3\psi^{+}(\z_1, \z_2, J\z_3)=0$. 

Assume that $b_3$ is not identically zero on $M$. Then $\psi^{+}(\z_1, \z_2, J\z_3)=0$ on some non-empty open set $U$. 
On the other hand, the $1$-form $\psi^{+}(\z_1, \z_2,\cdot)$ vanishes when applied to $\z_1$, $J\z_1$, $\z_2$, $J\z_2$ and $\z_3$, so by Lemma \ref{2.3}, $\psi^{+}(\z_1, \z_2,\cdot)$ vanishes on the non-empty open set $U\cap M_0$. This contradicts Corollary \ref{c1}. Consequently $b_3\equiv 0$, and similarly $b_2= b_1 \equiv 0$. We thus get 
\beq\label{12345}\z_4=\sum_{j=1}^3a_j\z_j.\eeq
Taking the Lie derivative in \eqref{12345} with respect to $\z_i$ and $J\z_i$ for $i=1,2,3$ and using Lemma \ref{start} (iii) we obtain $\z_i(a_j)=J\z_i(a_j)=0$ for every $i,j\in\{1,2,3\}$, so $a_j$ are constant on $M_0$, thus showing that $\z_4$ is a linear combination of $\z_1,\z_2,\z_3$, a contradiction. 
\r

\section{Toric NK structures}

In view of Corollary \ref{c35} we can now introduce the following:

\begin{defi}
A $6$-dimensional strict NK manifold is called toric if its automorphism group has rank $3$, or equivalently, if it carries $3$ linearly independent mutually commuting pseudo-holomorphic Killing vector fields $\z_1,\z_2,\z_3$.
\end{defi}

Assume from now on that $(M^6,g,J,\z_1,\z_2,\z_3)$ is a toric NK manifold and consider on the dense open subset $M_0$ given by Lemma \ref{2.3} the basis $\{ \theta^1, \theta^2, \theta^3,
\gamma^1,\gamma^2,\gamma^3 \}$ of $\Lambda^1M_0$ dual 
to $\{\z_1,\z_2,\z_3, J\z_1, J\z_2, J\z_3 \}$, together with the function
\begin{equation} \label{eps}\varepsilon:=\psi^{-}(\z_1, \z_2, \z_3).
\end{equation} 
For further use, let us
also introduce the symmetric $3\times 3$ matrix 
\beq\label{mc}C:=(C_{ij})=(g(\z_i,\z_j)).\eeq

As a direct consequence of Lemma \ref{2.3}, we have that $\z +J\z=\T M_0$,
where $\z$ is the $3$-dimensional distribution
spanned by $\z_k, 1 \le k \le 3$. This enables us to express $\psp$,
and $\psm$ in terms of the basis $\{ \theta^i, \gamma^j \}$ and of the
function $\varepsilon$, simply by checking that the two terms are
equal when applied to elements of the basis $\{\z_i, J\z_j\}$ of $\T M_0$:
\begin{equation} \label{psis}
\begin{split}
\psi^{+}=&\varepsilon(\gamma^{123}- \theta^{12} \wedge \gamma^3-\theta^{31}
\wedge \gamma^2-\theta^{23} \wedge \gamma^1),\\ 
\psi^{-}=& \varepsilon(\theta^{123}-\gamma^{12} \wedge \theta^3-\gamma^{31}
\wedge \theta^2-\gamma^{23} \wedge \theta^1),
\end{split}
\end{equation}
where here and henceforth the notation $\gamma ^{123}$ stands for
$\gamma ^1\wedge \gamma ^2\wedge \gamma ^3$ etc.
Recalling the definition of $\mu_{ij}:=\omega(\z_i, \z_j)$, the fundamental 2-form $\omega:=g(J\cdot,\cdot)$ can be expressed by the formula:
\begin{equation}\label{omega}
\omega=\sum \limits_{1 \le i < j \le 3
}\mu_{ij}(\theta^{ij}+\gamma^{ij})+\sum \limits_{i=1}^3  \theta^i
\wedge c^i 
\end{equation}
where the $1$-forms $c^i$ in $\Lambda^1(J\z^*)$ are given by $c^i=\sum
\limits_{j=1}^{3} C_{ij}\gamma^j$. A short computation yields  
\begin{equation}\label{omega3}
\omega^3=-6 \theta^{123} \wedge c^{123}+6
\theta^{123}\wedge c\wedge\eta, 
\end{equation}
where $\eta$ in $\Lambda^2(J\z^*)$ is given by 
$$\eta:=\sum \limits_{1 \le
  i<j\le 3} \mu_{ij} \gamma^{ij}$$ 
and $c$ in $\Lambda^1(J\z^*)$ is given
by 
$$c:=\mu_{23}c^1+\mu_{31}c^2+\mu_{12}c^3.$$
Therefore from the compatibility relations \eqref{comp} it follows that 
\begin{equation} \label{vol}
c^{123}={\varepsilon^2}\gamma^{123}+c\wedge\eta,
\end{equation}
which is equivalent to
\begin{equation} \label{vol1} \det C=\varepsilon^2+^t\!V CV, 
\end{equation}
where we denote by
\begin{equation} \label{vv}V:=\biggl (\begin{array}{c}
\mu_{23}\\
\mu_{31}\\
\mu_{12}
\end{array} \biggr ).
\end{equation}
\begin{lema} \label{step3}
The following relations hold:
\begin{itemize}
\item[(i)] $\d \mu_{12}=-3\varepsilon \gamma^3,\ \d \mu_{31}=-3\varepsilon
  \gamma^2,\ \d \mu_{23}=-3\varepsilon \gamma^1$; 
\item[(ii)] $ \d \varepsilon=4c.$
\end{itemize}
\end{lema}
\begin{proof}
(i) Using \eqref{nk1}, \eqref{psis} and the Cartan formula we can write
$$\d \mu_{12}=\d (\z_2\i\z_1\i\omega)=\z_2\i\z_1\i \d \omega=3\z_2\i\z_1\i
\psp=-3\varepsilon \gamma^3.$$ 
The other formulas are similar.

(ii) Using \eqref{nk2}, \eqref{omega} and the Cartan formula again, we get
\bea \d \varepsilon&=&\d (\z_3\i\z_2\i\z_1\i\psm)=-\z_3\i\z_2\i\z_1\i
\d \psm\\
&=&2\z_3\i\z_2\i\z_1\i\omega
^2=4(\mu_{23}c^1+\mu_{31}c^2+\mu_{12}c^3).
\eea

\end{proof}
We will now show that Equation \eqref{nk2} is equivalent to 
some exterior system involving the 1-forms $\theta ^i$.
\begin{lema} \label{thetas}
Equation \eqref{nk2} holds if and only if the
forms $\theta_i, 1 \le i 
\le 3$ satisfy the differential system: 
\begin{equation}\label{sy}
\begin{split}
\frac{1}{4}\varepsilon \d \theta^1=& c^2 \wedge c^3-\mu_{23} \eta \\
\frac{1}{4}\varepsilon \d \theta^2=& c^3 \wedge c^1-\mu_{31} \eta \\
\frac{1}{4}\varepsilon \d \theta^3=& c^1 \wedge c^2-\mu_{12} \eta 
\end{split}
\end{equation}
\end{lema}
\begin{proof}

Assume that \eqref{nk2} holds. By \eqref{psis}
\begin{equation}\label{ep}
\z_2 \lrcorner \z_1 \lrcorner \psi^{-}=\varepsilon \theta^3.
\end{equation}
Since $\z_k, 1 \le k \le 3$ are commuting Killing vector fields
preserving the whole $\SU(3)$-structure, 
\eqref{omega} yields
$$ \d (\z_2 \lrcorner \z_1 \lrcorner \psi^{-})=\z_2 \lrcorner \z_1
\lrcorner \d \psi^{-}=-2 \z_2 \lrcorner \z_1
\lrcorner (\omega \wedge \omega)=-4\theta ^3\wedge
c-4\mu_{12}\eta+4c^1\wedge c^2.
$$
hence by \eqref{ep} and Lemma \ref{step3} (ii) we get
$$\frac{1}{4}\varepsilon \d \theta^3=\frac{1}{4}\d (\varepsilon
\theta^3)-\frac{1}{4}\d \varepsilon \wedge\theta^3 =-\theta ^3\wedge
c-\mu_{12}\eta+c^1\wedge c^2-c\wedge\theta ^3=c^1\wedge
c^2-\mu_{12}\eta.$$
The proof of the two other relations is similar.

Conversely, we notice that \eqref{nk2} holds if and only if 
$$\begin{cases}
\z_i \lrcorner \z_j \lrcorner \d \psi^{-}=-2\z_i \lrcorner \z_j
\lrcorner \omega ^2,\qquad\forall\ 1\le i,j\le 3,\\
J\z_1\i J\z_2\i J\z_3\i \d \psi^{-}=-2J\z_1\i J\z_2\i J\z_3\i \omega ^2.
\end{cases}
$$
The first relation was just shown to be equivalent to \eqref{sy}. It
remains to check, by a straightforward calculation, that the second
relation is automatically fulfilled. 

\end{proof}
We finally interpret 
Equation \eqref{nk1} in terms of the frame $\{c^i\}$.   
\begin{lema} \label{cs} Equation \eqref{nk1} holds if and only if
  \eqref{vol} holds and the
  forms $\varepsilon c^k$ are closed for $1 \le k \le 3$. 
\end{lema}
\begin{proof} 
Taking the interior product with $\z_1$ in \eqref{nk1} and using
\eqref{psis}, \eqref{omega} and Lemma \ref{step3} (i) yields
\bea 3\varepsilon(-\theta ^2\wedge\gamma ^3+\theta ^3\wedge\gamma
^2)&=&3\z_1\i\psp=\z_1\i \d \omega=-\d (\z_1\i\omega)=-\d (\mu_{12}\theta
^2-\mu_{31}\theta ^3+c^1)\\
&=&3\varepsilon\gamma ^3\wedge\theta ^2 -\mu_{12}\d \theta ^2
-3\varepsilon\gamma ^2\wedge\theta ^3 +\mu_{31}\d \theta ^3-\d c^1,
\eea
whence 
$$\d c^1=\mu_{31}\d \theta ^3 -\mu_{12}\d \theta ^2.$$
From Lemma \ref{thetas} and  Lemma \ref{step3} (ii) we thus obtain
\bea \d (\varepsilon c^1)&=&4c\wedge c^1+4[\mu_{31}(c^1 \wedge c^2-\mu_{12}\eta)-
\mu_{12}(c^3 \wedge c^1-\mu_{31}\eta)]\\
&=&4(\mu_{23}c^1+\mu_{31}c^2+\mu_{12}c^3)\wedge c^1+4(\mu_{31}c^1 \wedge c^2-
\mu_{12}c^3 \wedge c^1)=0.
\eea

Conversely, we notice that \eqref{nk1} holds if and only if 
$$\begin{cases}
\z_i \lrcorner \d \omega=3\z_i 
\lrcorner \psp,\qquad\forall\ 1\le i\le 3,\\
J\z_1\i J\z_2\i J\z_3\i \d \omega=3J\z_1\i J\z_2\i J\z_3\i \psp.
\end{cases}
$$
We have just shown that the first equation is equivalent to
$\varepsilon c^k$ being closed. 
The component of $\d  \omega=3\psi^{+}$ on $\Lambda^3 J\z$ is given by 
$$ \d \eta+\sum \limits_{k=1}^3 \d \theta^k \wedge c^k=3\varepsilon \gamma^{123},
$$
so using \eqref{sy}, the second equation is equivalent to \eqref{vol1}.
\end{proof}

Let us now consider the 3-dimensional quotient $U:=M_0\slash \z$ of the open set
$M_0$ by the action of the $3$-dimensional torus generated by the
Killing vector fields $\z_i$. Clearly the natural projection $\pi:M\to U$ is a
submersion. We shall now interpret the 
geometry of the situation down on $U$. Since $\z_i(\mu_{jk})=0$, there
exist functions $y_i$ on $U$ such that $\pi ^*y_1=\mu_{23},\pi
^*y_2=\mu_{31},\pi ^*y_3=\mu_{12}$. Moreover, since $\varepsilon$ does
not vanish on $M_0$, Lemma \ref{step3} (i) shows that $\{y_i\}$ define
a global coordinate system on $U$. From now on we will identify the
projectable functions or exterior forms on $M$ with their projection
on $U$. Since everything is local, we may suppose that $U$ is contractible. 

\begin{rema}
By Lemma \ref{start} (i) it follows that the map $\mu:M\to\Lambda^2\RM^3\cong\mathfrak{so}(3)$ defined by 
$$ \mu:=  \begin{pmatrix}
0 & \mu_{12} & \mu_{13} \\
\mu_{21} & 0 & \mu_{23} \\
\mu_{31} & \mu_{32} & 0
\end{pmatrix}=\pi^*\begin{pmatrix}
0 & y_3 & -y_2 \\
-y_3 & 0 & y_1  \\
y_2 & -y_1 & 0
\end{pmatrix}
$$
is the multi-moment map of the strong geometry $(M,\psi^+)$ defined by Madsen and Swann in \cite{ms} and studied further by Dixon \cite{d} in the particular case where $M=S^3\times S^3$. Similarly, the function $\varepsilon$ can be seen as the multi-moment map associated to the closed 4-form $\d\psi^-$.
These maps will play an important role in Sections 5 and 6 below. 
\end{rema}

\begin{pro}
There exists a function $\f$ on $U$ (defined up to an affine function)
such that 
$\Hess(\f)=C$ in the coordinates $\{y_i\}$.
\end{pro}

\begin{proof}
From Lemma \ref{cs}, there exist functions $f_i$ on $U$ such that
$\d f_i=\varepsilon c^i$ for $1\le i\le 3$. Notice that by Lemma
\ref{step3} (i), this is
equivalent to 
\beq\label{de}\frac{\del f_i}{\del y_j}=-\frac13 C_{ij}.
\eeq
From Lemma \ref{step3} (i)
we get
$$\d (\sum_{i=1}^3f_i\d y_i)= \sum_{i=1}^3\d f_i\wedge \d y_i=-3
\sum_{i=1}^3\varepsilon c^i\wedge \varepsilon \gamma^i=
\sum_{i,j=1}^3\varepsilon ^2C_{ij}\gamma^j\wedge \gamma^i=0,$$
so there exists some function $\f$ such that 
$$\d \f=-3\sum_{i=1}^3f_i\d y_i.$$
This means that $\frac{\del \f}{\del y_i}=-3f_i$, which together with
\eqref{de} finishes the proof.
\r

Let us introduce the operator $\dr$ of radial differentiation, acting
on functions on $U$ by 
$$\dr f:=\sum_{i=1}^3 y_i\frac{\del f}{\del y_i}.$$

\begin{pro} The function $\f$ can be chosen in such a way that
\beq\label{f} \varepsilon ^2=\frac83(\f-\dr \f) .\eeq
\end{pro}
\begin{proof}
It is clearly enough to show that the exterior derivatives of the two
terms coincide. Since 
$$\frac{\del (\dr\f)}{\del y_j}=\sum_{i=1}^3\frac{\del^2\f}{\del y_i\del
  y_j}y_i+\frac{\del \f}{\del y_j},$$
Lemma \ref{step3} yields
$$-\frac83\d (\dr \f-\f)=-\frac83\sum_{i,j=1}^3
C_{ij}y_i\d y_j=8\sum_{i,j=1}^3 C_{ij}y_i\varepsilon \gamma
^j=8\varepsilon c=\d (\varepsilon ^2).$$
\r

Summing up, we get the following result:

\begin{coro}
The function $\f$ given in the previous proposition satisfies the
equation
\begin{equation} \label{ma} \det
  (\Hess(\f))=\frac83\f-\frac{11}3\dr\f+\dr ^2\f.
\end{equation}
\end{coro}

\begin{proof} We have 
\beq\label{ns}\dr ^2\f=\dr\bigg(\sum_{i=1}^3
y_i\frac{\del \f}{\del y_i}\bigg)=\sum_{i=1}^3
y_i\frac{\del \f}{\del y_i}+\sum_{i,j=1}^3 y_iy_j\frac{\del^2\f}{\del
  y_i\del y_j} 
=\dr\f+^t\!VCV,\eeq
so \eqref{ma} is a consequence of \eqref{vol1} and \eqref{f}. 
\r
\bigskip

\section{The inverse construction}
\medskip

In this section we will show that conversely, every solution $\f$ of
Equation \eqref{ma} on some open set $U\subset \RM^3$ defines 
a NK structure with 3 linearly independent
commuting Killing vector fields on $U_0\times\TM^3$, where $U_0$ is
some open subset of $U$. More precisely, let $y_1,y_2,y_3$ be the
standard coordinates on $U$ and let $\mu$ be the $3\times 3$
skew-symmetric matrix 
\beq\label{mmu} \mu:= \begin{pmatrix}
0 & y_3 & -y_2 \\
-y_3 & 0 & y_1  \\
y_2 & -y_1 & 0
\end{pmatrix}. 
\eeq
Define the $6\times 6$ symmetric matrix 
$$D:=\begin{pmatrix}\Hess(\f)&-\mu\cr  \mu &\Hess(\f)\end{pmatrix}.$$
Let $U_0$ denote the open set 
\beq\label{up}
U_0:=\{x\in U\ |\ \f(x)-\dr\f(x)>0\ \hbox{and}\ D\ \hbox{is
  positive definite}\}.
\eeq
The next result is straightforward:
\begin{lema} \label{pos-big}
The matrix $D$ is positive definite if and only if 
\begin{itemize}
\item[(i)] $C=\Hess(\f)$ is positive definite and
\item[(ii)] $\langle \mu a,b\rangle^2 <\langle Ca,a \rangle \langle Cb,b \rangle$
for all $(a,b)\in(\mathbb{R}^3\times \mathbb{R}^3)\setminus{(0,0)}$.
\end{itemize}
\end{lema}
On $U_0$ we define a positive function $\varepsilon$ by \eqref{f},
1-forms $\gamma ^i$ by $\d y_i=-3\varepsilon \gamma ^i$ and a 2-form
$\eta:=y_1\gamma ^2\wedge\gamma ^3+y_2\gamma ^3\wedge\gamma
^1+y_3\gamma ^1\wedge\gamma ^2$. We denote 
as before by $C$ the Hessian of $\f$ and define $c^i:=\sum_{j=1}^3
C_{ij}\gamma ^j$. 

\begin{lema} \label{3.1} The following hold:
\begin{itemize}
\item[(i)] The $1$-forms $\varepsilon c^i$ are exact.
\item[(ii)] The $2$-forms $\tau_1:=(c^2 \wedge c^3-y_1
  \eta)/\varepsilon$, $\tau_2:=(c^3 \wedge 
c^1-y_2 \eta)/\varepsilon$ and $\tau_3:=(c^1 \wedge c^2-y_3 \eta)/\varepsilon$
are closed.
\end{itemize}
\end{lema}
\begin{proof}
(i) We have:
$$\d \bigg(-\frac13\frac{\del \f}{\del
  y_i}\bigg)=-\frac13\sum_{j=1}^3\frac{\del ^2 
  \f}{\del y_i\del y_j}\d y_j =-\frac13\sum_{j=1}^3C_{ij}\d y_j=\varepsilon c^i.$$

(ii) We first compute using (i):
\bea \d (\varepsilon ^3\tau_1)&=&\d (\varepsilon ^2(c^2 \wedge c^3-y_1
\eta))= -\d (y_1\varepsilon ^2\eta)\\
&=&-\d (y_1^2\varepsilon ^2\gamma ^{23}+y_1y_2\varepsilon ^2\gamma
^{31}+y_1y_3\varepsilon ^2\gamma ^{12})=12y_1\varepsilon ^3\gamma
^{123}.
\eea
On the other hand, 
\bea \d (\varepsilon ^3)\wedge\tau_1&=&3\varepsilon ^2
\d \varepsilon\wedge\tau_1=12\varepsilon(\sum_{j=1}^3y_jc^j)\wedge(c^2
\wedge c^3-y_1 \eta)\\
&=&12\varepsilon y_1(\det C-\sum_{i,j=1}^3C_{ij}y_iy_j)\gamma
^{123}=12y_1\varepsilon ^3\gamma 
^{123},
\eea
the last equality (which is the converse to \eqref{vol1}) following
from \eqref{f}, \eqref{ma} and \eqref{ns}. 
These two relations show that $\tau_1$ is closed. The proof that
$\d \tau_2=\d \tau_3=0$ is similar.
\r

By replacing $U_0$ with a smaller open subset if necessary, one can find
1-forms $\sigma_i$ such that $\d \sigma_i=4\tau_i$.
Consider now the 6-dimensional manifold $M:=U_0\times \TM^3$ with
coordinates $y_1,y_2,y_3$ and $x_1,x_2,x_3$ (locally defined). The
1-forms $\theta 
^i:=\d x_i+\sigma _i$ 
satisfy the differential system \eqref{sy}. We define
$\psi ^\pm$ and $\omega$ by \eqref{psis} and 
\eqref{omega} and we claim that they determine a strict NK structure
on $M$ whose automorphism group contains a 3-torus. 

Let us first check that $(\psi ^\pm,\omega)$ satisfy the conditions of
Lemma \ref{1.1}. The relation (i) is straightforward, (ii) is
equivalent to \eqref{vol1}, and (iii) holds from the definition
\eqref{up} of $U_0$. 

In order to prove that $(\psi ^\pm,\omega)$ defines a NK structure, we
need to check \eqref{nk1} and \eqref{nk2}. By Lemma \ref{cs},
\eqref{nk1} is equivalent to $\varepsilon c^i$ being closed (Lemma
\ref{3.1} (i)) together with \eqref{vol1}. Similarly, Lemma
\ref{thetas} shows that \eqref{nk2} is equivalent to the system
\eqref{sy} together with \eqref{vol1} again.

It remains to check that the automorphism group contains a
3-torus. This is actually clear: the action of $\TM^3$ on
$M=U_0 \times \TM^3$ by multiplication on the
first factor, preserves the $\SU(3)$ structure.
We have proved the following result:

\begin{teo}
Every solution of the
Monge--Amp\`ere-type equation 
\eqref{ma} on some open set $U$ in $\RM^3$ defines in a
canonical way a NK structure with $3$ linearly independent commuting
infinitesimal automorphisms on $U_0\times \TM^3$, where $U_0$ is defined
by \eqref{up}. 
\end{teo}

\section{Examples}

We will illustrate the above computations on a specific example of
toric nearly K\"ahler manifold, namely the 3-symmetric space
$S^3\times S^3$.

Let $K:=\SU_2$ with Lie algebra $\kk = \su_2$ and $G:=K\times K \times K
$ with Lie algebra 
$\gg =	\kk\oplus \kk \oplus \kk$. We consider 
the $6$-dimensional manifold $M = G/K$, where $K$ is diagonally
embedded in $G$. The tangent space of $M$
at $o=eK$ can be identified with
$$
\pp = \{(X,Y,Z)\in \kk \oplus \kk \oplus \kk \,|\, X+Y+Z=0\}  .
$$
Consider the invariant scalar product $B$ on $\su_2$ such that the
scalar product 
$$ 
<(X,Y,Z),(X,Y,Z)> := B(X,X) + B(Y,Y) + B(Z,Z)
$$
defines the homogeneous nearly K\"ahler metric $g$ of scalar  
curvature $30$ on $M=S^3\times S^3$ (cf. \cite{au}, Lemma 5.4).

The $G$-automorphism $\sigma$ of order 3 defined by
$\sigma(a_1,a_2,a_3) =(a_2,a_3,a_1)$ induces a canonical almost
complex structure on the $3$-symmetric space $M$ by the relation 
$$\sigma=\frac{-\Id+\sqrt 3 J}{2},\qquad\hbox{on}\ \pp,$$
whence

\beq\label{js}J(X,Y,Z) = \tfrac{2}{\sqrt 3} (Y,Z,X) +
\tfrac{1}{\sqrt{3}}(X,Y,Z),\qquad\forall(X,Y,Z)\in\pp. 
\eeq

Let $\xi$ be a unit vector in $\su_2$ with respect to $B$. The
right-invariant vector 
fields on $G$ generated by the elements
$$\tilde\zeta_1=(\xi,0,0),\qquad \tilde\zeta_2=(0,\xi,0),\qquad
\tilde\zeta_3=(0,0,\xi)$$
of $\gg$, define three commuting Killing vector fields $\zeta_1$,
$\zeta_2$, $\zeta_3$ on $M$. 

Let us compute $g(\zeta_1,J\zeta_2)$ at some point $aK\in M$, where
$a=(a_1,a_2,a_3)$ is some element of $G$. By the definition of $J$ we have
\bea g(\zeta_1,J\zeta_2)_{aK}&=&<(a ^{-1}\tilde\zeta_1 a)_{\pp},J(a
^{-1}\tilde\zeta_2 a)_{\pp}>=<(a_1 ^{-1}\xi a_1,0,0)_{\pp},J(0,a_2
^{-1}\xi a_2,0)_{\pp}>\\
&=&\frac19<(2a_1 ^{-1}\xi a_1,-a_1 ^{-1}\xi a_1,-a_1 ^{-1}\xi
  a_1),J(-a_2^{-1}\xi a_2,2a_2^{-1}\xi a_2,-a_2^{-1}\xi
a_2)>\\
&=&\frac{1}{9}<(2a_1 ^{-1}\xi a_1,-a_1 ^{-1}\xi a_1,-a_1 ^{-1}\xi
  a_1),\sqrt3 (a_2^{-1}\xi a_2,0,-a_2^{-1}\xi
a_2)>\\
&=&\frac{1}{\sqrt3}B(a_1 ^{-1}\xi a_1,a_2^{-1}\xi a_2).
\eea

We introduce the functions
$y_1,y_2,y_3:G\to\RM$ defined by 
$$y_i(a_1,a_2,a_3)=-\frac{1}{\sqrt3}B(a_j^{-1}\xi a_j,a_k^{-1}\xi a_k),\qquad
$$
for every permutation $(i,j,k)$ of $(1,2,3)$. The previous computation
yields: 
$$g(J\zeta_2,\zeta_3)_{aK}=y_1(a),\qquad
g(J\zeta_3,\zeta_1)_{aK}=y_2(a),\qquad
g(J\zeta_1,\zeta_2)_{aK}=y_3(a),
\qquad\forall a\in G.$$

A similar computation yields 
$$g(\zeta_i,\zeta_j)_{aK}=\frac23\delta_{ij}+\frac{1}{\sqrt 3}y_k(a)$$
for every even permutation $(i,j,k)$ of $(1,2,3)$. In other words, the
matrix $C$ defined in \eqref{mc} satisfies
$$C_{ij}=\frac23\delta_{ij}+\frac{1}{\sqrt 3}y_k,$$
where by a slight abuse of notation we keep the same notations $y_i$
for the functions defined on $M$ by the $K$-invariant functions $y_i$
on $G$.

The function $\f$ in the coordinates $y_i$ such that $\Hess(\f)=C$ is
determined by
\beq\label{phi}\f(y_1,y_2,y_3)=\frac13(y_1^2+y_2^2+y_3^2)+\frac{1}{\sqrt
  3}y_1y_2y_3+h,
\eeq
up to some affine function $h$ in the coordinates $y_i$. On the other
hand, since
$$\det(C)=-\frac29(y_1^2+y_2^2+y_3^2)+\frac{2}{3\sqrt 3}y_1y_2y_3+\frac8{27},$$
an easy computation shows that the function $\f$ given by \eqref{phi}
satisfies indeed the 
Monge--Amp\`ere-type equation \eqref{ma} for $h=\frac19$. For the sake of
completeness we list the other functions involved in the
previous section, in the particular case of the present situation:
$$\varepsilon ^2= -\frac89(y_1^2+y_2^2+y_3^2)-\frac{16}{3\sqrt 3}y_1y_2y_3+\frac8{27},$$
$$^t\!VCV=\frac23(y_1^2+y_2^2+y_3^2)+2\sqrt
  3y_1y_2y_3,$$
  where $\varepsilon$ was defined in \eqref{eps} and $V$ in \eqref{vv}.

\subsection{Radial solutions} \label{rad}
We search here particular solutions to equation \eqref{ma}, namely when  $\varphi$ is a radial function on (some open subset of) $\mathbb{R}^3$ with coordinates $y_k, 1 \le k \le 3$. Let therefore $\varphi(y_1,y_2,y_3):=x(\frac{r^2}{2})$ where $x$ is a function 
of one real variable and $r^2=y_1^2+y_2^2+y_3^2$. 
A direct computation yields 
\begin{equation*}
\begin{split}
\Hess(\varphi)=&\,\begin{pmatrix}
y_1^2 x^{\prime \prime}+x^{\prime}& y_1y_2 x^{\prime \prime} &  y_1y_3 x^{\prime \prime}\\
y_1y_2 x^{\prime \prime} & y_2^2 x^{\prime \prime}+x^{\prime} & y_2y_3 x^{\prime \prime}\\
 y_1y_3 x^{\prime \prime} &  y_2y_3 x^{\prime \prime} & y_3^2 x^{\prime \prime}+x^{\prime}
\end{pmatrix}\\
=&\,x^{\prime}\Id+x^{\prime \prime}(\frac{r^2}{2}) V \cdot {}^{t}V
\end{split}
\end{equation*}
where $V:=\begin{pmatrix}
y_1\\
y_2\\
y_3 \end{pmatrix}$. In particular,
\begin{equation*}
\begin{split}
\det \Hess(\varphi)=&\,(x^{\prime})^2 x^{\prime \prime}r^2+(x^{\prime})^3\\
\partial_r \varphi=&\,r^2 x^{\prime}, \  \partial_r^2 \varphi=r^4 x^{\prime \prime}+2r^2x^{\prime},
\end{split}
\end{equation*}
whence after making the substitution $t:=\frac{r^2}{2}$ we get:
\begin{pro} \label{examples} Radial solutions to the Monge-Amp\`ere type equation \eqref{ma} are given by solutions of the second order $ODE$ 
\begin{equation} \label{ode}
x^{\prime \prime}=F(t,x,x^{\prime})
\end{equation}
where $F(t,p,q):=\frac{8p-(10tq+3q^3)}{6(q^2t-2t^2)}$.
\end{pro}  
To decide which solutions to \eqref{ode} yield genuine Riemannian metrics in dimension six we observe that 
\begin{pro} \label{rad-no}
For a radial solution $\varphi=x(\frac{r^2}{2})$ to \eqref{ma}, the set $U_0$ defined in \eqref{up} is
$$U_0=\{ t>0\ |\  x(t)>2tx^{\prime}(t)> 2t\sqrt{2t} \}. $$
\end{pro}
\begin{proof}
Having $\f-\partial_r \f>0$ is equivalent with 
$$ 2tx^{\prime}(t)-x(t)<0.
$$
The matrix $\Hess(\f)$ has the eigenvalues $x^{\prime}(\frac{r^2}{2})$ with eigenspace $E:=\{a \in \mathbb{R}^3\ |\  \langle a,y\rangle=0\}$
and $x^{\prime}(\frac{r^2}{2})+r^2 x^{\prime \prime}(\frac{r^2}{2})$ with eigenvector $y$. Therefore $\Hess(\varphi)>0$ if and only if 
\beq\label{ine} x^{\prime}(t)>0, \ x^{\prime}(t)+2tx^{\prime \prime}(t)>0.
\eeq
However $x^{\prime}(t)+2tx^{\prime \prime}(t)=\frac{8(x-2tx^{\prime})}{3((x^{\prime})^2-2t)}$ from \eqref{ode}, thus showing that the system \eqref{ine} is equivalent to $x^{\prime}(t) >\sqrt{2t}$. By Lemma \ref{pos-big}, it remains to interpret the condition 
\beq\label{cond}\langle \mu a,b\rangle^2 <\langle Ca,a \rangle \langle Cb,b \rangle\eeq
for all $(a,b)\in(\mathbb{R}^3\times \mathbb{R}^3)\setminus{(0,0)}$. 

We split $a=\lambda_1 y+v_1$, $b=\lambda_2 y+v_2$, with $v_1,v_2\in E$ and take into account that $C$ preserves the
orthogonal decomposition $\mathbb{R}^3=\mathbb{R}y \oplus E$ and also that  $y$ belongs to $\ker \mu$. Then
$$  \langle Ca,a \rangle \langle Cb,b \rangle=(\lambda_1^2\langle Cy,y\rangle+\langle Cv_1, v_1 \rangle )(\lambda_2^2\langle Cy,y\rangle+\langle Cv_2, v_2 \rangle )$$
and since $\mu$ is skew-symmetric,
$$\langle \mu a,b\rangle^2=\langle \mu v_1, v_2\rangle^2.
$$
Thus \eqref{cond}
holds if and only if $\langle Cv_1, v_1 \rangle \langle Cv_2,v_2 \rangle > \langle \mu v_1,v_2\rangle^2 $ for all non-zero $v_1,v_2\in E$. This is equivalent to 
\beq\label{iu}\langle \mu v_1,v_2\rangle^2 <(x^{\prime}(t))^2 \vert v_1 \vert^2 \vert v_2 \vert^2
\eeq
for all $v_1,v_2$ in $E\setminus\{0\}$. By the Cauchy-Schwartz inequality this is equivalent to $-\frac12\tr(\mu^2) <(x^{\prime})^2 (t)$ and since $\tr(\mu^2)=-2r^2=-4t$, \eqref{iu} is equivalent to $x^{\prime}(t) >\sqrt{2t}$. However this was already known and 
the proof is finished.
\end{proof}
\begin{rema} \label{out1}
The solutions of the ODE \eqref{ode} of the form $x=kt^l$ with $k,l\in\RM$ are $x_{1,2}=\pm \frac{2\sqrt{2}}{9}t^{\frac{3}{2}}$ and $x_3=kt^\frac12$, corresponding to 
\begin{equation*} \label{spec2}
\varphi_{1,2}=\pm \frac{r^3}{9},\qquad \f_3=\frac{k}{\sqrt 2}r.
\end{equation*}
However, they do not satisfy the positivity requirements from Proposition \ref{rad-no}.
\end{rema}
Solutions to the Cauchy problem \eqref{ode}, admissible in the sense of Proposition \ref{rad-no}, are obtained by requiring the initial data $(t_0, x(t_0), x^{\prime}(t_0))$ belong to 
$$ \mathcal{S}:=\{(t,p,q) \in \mathbb{R}^3: t>0, \ p>2tq> 2t\sqrt{2t} \}. 
$$

\end{document}